\documentclass[11pt,twoside]{article}

\usepackage{enumerate}
\usepackage{graphics}
\usepackage{graphicx}
\usepackage{amssymb}
\usepackage{amsmath}
\usepackage{amsthm}
\usepackage{mathtools}
\usepackage{latexsym}
\usepackage[title,titletoc]{appendix}
\usepackage{color}
\usepackage{mathrsfs}
\usepackage{indentfirst}
\usepackage{txfonts}
\usepackage{anysize}
\usepackage{enumitem}
\usepackage{microtype}
\usepackage{booktabs}
\usepackage{longtable}
\usepackage{array}
\usepackage{url}

\usepackage[colorlinks=true,
linkcolor=blue,
citecolor=red,
urlcolor=magenta,
]{hyperref}

\usepackage{txfonts}
\usepackage{anysize}

\allowdisplaybreaks

\newtheorem{theorem}{Theorem}[section]
\newtheorem{lemma}[theorem]{Lemma}

\newtheorem{proposition}[theorem]{Proposition}
\theoremstyle{definition}

\numberwithin{equation}{section}

\begin{document}

\title{\vspace{-0cm}\bf\Large
A Solution of Problem 2.5 by Brezis on the Planar Ginzburg--Landau
Equation\footnotetext{\hspace{-0.35cm}
2020 \emph{Mathematics Subject Classification}. Primary 35J60;
Secondary 35B40, 35B45.
\endgraf \emph{Key words and phrases}. Ginzburg--Landau equation,
entire solution, finite potential energy, Kelvin inversion, Jacobi field.
\endgraf This project is partially supported by
the National Natural Science Foundation of China
(Grant Nos. 12431006, 12371093, and 12501118),
the Natural Science Foundation of Fujian Province (Grant No. 2026J008197),
the Beijing Natural Science Foundation (Grant No. 1262011),
the Fundamental Research Funds for the Central Universities
(Grant No. 2253200028), and Longyuan Young Talents of Gansu Province.}}
\author{Xiaosheng Lin, Dachun Yang\footnote{Corresponding
author, E-mail: \texttt{dcyang@bnu.edu.cn}/{\color{red}\today}/Final version.},
\ \ Sibei Yang, Wen Yuan
and Yangyang Zhang}
\date{}
\maketitle

\vspace{-0.8cm}

\begin{center}
\begin{minipage}{13.8cm}
{\small {\bf Abstract}\quad
Brezis, Merle, and Rivi\'ere [Arch. Rational Mech. Anal. 1994] proved that,
if a smooth solution
$u:\mathbb{R}^2\to\mathbb{C}$ of the planar entire Ginzburg--Landau equation
$$
-\Delta u = u(1-|u|^2)\quad \text{in}\quad \mathbb{R}^2
$$
satisfies the finite potential energy estimate that
$$
\int_{\mathbb{R}^2}\left[1-|u(x)|^2\right]^2\,dx < \infty,
$$
then it has the asymptotic property that
$$
|u(x)|\to 1\quad \text{as}\quad |x|\to\infty.
$$
The converse problem whether the asymptotic property implies
the finite potential energy estimate
was originally posed in their work and later formulated
by Brezis as \emph{Open Problem 2.5} in [Atti Accad. Naz. Lincei Rend. Lincei Mat. Appl. 2023].
In this article, we give an affirmative answer to this question.
Our proof relies on the Kelvin inversion, a Morrey-type energy decay
for the translation Jacobi system,
the $L^4$-integrability of the phase form,
and a linearised amplitude equation; all these tools
together imply the $L^2$-integrability of the function $1-|u|$ on an exterior domain, without
any a priori integrability assumption.}
\end{minipage}
\end{center}

\tableofcontents

\section{Introduction and main results}

The Ginzburg--Landau equation, arising in the description of macroscopic stationary states of superfluids,
Bose--Einstein condensates, and solitary waves in plasmas, has been a subject of sustained investigation
over the past several decades (see, for example, \cite{bbo02,bos06,b94,bmr94,l98,l96,l95,lr01,pr00})
and continues to attract considerable interest in modern analysis and
partial differential equations (see, for example, \cite{bd24,bd23,bbh17,ddmr22,djm25,e13,lmww25}).
In particular, the planar entire Ginzburg--Landau equation is
\begin{equation}\label{eq:GL}
-\Delta u=u\left(1-|u|^2\right)\quad\hbox{in}\quad\mathbb{R}^2,
\end{equation}
where $u:\mathbb{R}^2\to\mathbb{C}$ is smooth.
A classical result of Brezis, Merle, and Rivi\'ere \cite{bmr94} proved that, if $u$ satisfies \eqref{eq:GL}
and the finite potential energy estimate that
\begin{equation}\label{e1.2}
\int_{\mathbb{R}^2}\left[1-|u(x)|^2\right]^2\,dx < \infty,
\end{equation}
then it has the asymptotic property that
\begin{equation}\label{eq:modulus-limit}
|u(x)|\rightarrow1\qquad\hbox{as }|x|\to\infty.
\end{equation}

Conversely, Brezis, Merle, and Rivi\'ere \cite[Problem 2]{bmr94} asks whether the asymptotic
property \eqref{eq:modulus-limit} always implies the finite potential energy estimate
\eqref{e1.2} for the smooth solution $u$ to \eqref{eq:GL}; Brezis \cite{Brezis2023}
later posed this as \emph{Open Problem 2.5} in his ``favorite open problems" (see also \cite[Open Problem 2]{b99}).

In this article, we give an affirmative solution of this problem.

\begin{theorem}\label{thm:main}
If the solution $u\in C^\infty(\mathbb{R}^2;\mathbb{C})$ to \eqref{eq:GL} has
the asymptotic property \eqref{eq:modulus-limit}, then it also satisfies
the finite potential energy estimate \eqref{e1.2}.
\end{theorem}

We now describe the proof strategy for Theorem \ref{thm:main}.
The proof of Theorem \ref{thm:main} can be reduced to establishing two key estimates:
(i) the global bound $|u|\le 1$ on $\mathbb{R}^2$ (via the maximum principle);
(ii) the $L^2$-integrability of the function $h:=1-|u|$ on an exterior domain
$\{x\in\mathbb{R}^2:|x|>R_0\}$. The latter is the real challenge,
as no integrability is assumed a priori.

The main strategy of proving (ii) is to convert the asymptotic property into
a quantitative decay by studying a linearised system. On an exterior domain,
we decompose the solution $u=\rho\sigma$ into the product
of the amplitude $\rho$ and the phase $\sigma$, and we introduce the phase gradient $p$.
Linearising the amplitude equation gives a coupled Jacobi system for
$h$ and $p$, which admits a strong energy identity. Using the \emph{Kelvin inversion},
this exterior problem is transformed into a singular system on the ball $B_r$
with radius $r\in (0,\infty)$
small center at   the origin $\mathbf{0}$,
where a careful energy estimate yields a Morrey-type decay
$$
E(r)\le C_\beta r^{2\beta}
$$
for an energy $E(r)$ involving derivatives of $h$ and $p$.
This decay, together with a Gagliardo--Nirenberg interpolation
inequality, further implies the
$L^4$-integrability of $p$ on an exterior domain. Substituting this into
the following linear equation for $h$
$$
-\Delta h + \rho(1+\rho)h = \rho |p|^2
$$
further yields the $L^2$-integrability of $h$.
Together with $|u|\le 1$ on $\mathbb{R}^2$, this immediately
gives $1-|u|^2\in L^2(\mathbb{R}^2)$, proving the desired finite energy estimate.

The remainder of this article is organized as follows.
In Section \ref{s3}, we prove  Theorem~\ref{thm:main}
by first establishing two key propositions: the maximum modulus estimate of the solution $u$ to \eqref{eq:GL}
(Proposition~\ref{lem:maxmod}) and the $L^2$-integrability of the function $h:=1-|u|$ on
an exterior domain (Proposition~\ref{prop:h-L2}). Section~\ref{s4}
establishes several asymptotic properties for the solution $u$
and its derivatives, and derives the aforementioned amplitude--phase decomposition
and the associated system of equations on an exterior domain.
In Section~\ref{s5}, we study the Jacobi system for the derivatives
of the amplitude and the phase.
Specifically,
we prove a Caccioppoli-type inequality
and obtain a finite tail energy estimate. Section~\ref{s6} is devoted
to the Kelvin inversion argument: after transforming the exterior problem
into a singular system near
the origin, we establish a Morrey-type decay for energy involving derivatives of $h$ and $p$.
Section~\ref{s7} exploits this decay to prove the $L^4$-integrability of the phase form.
Finally, Section~\ref{s8} completes the proof of Proposition~\ref{prop:h-L2} by combining
the preceding estimates with a linearised amplitude equation.

We end this introduction by making some notational conventions. Throughout this
article, we \emph{always} denote by $C$ a positive constant which is independent
of the main parameters involved, but it may vary from line to line. We also use the
\emph{symbol} $C_{\alpha,\beta,\ldots}$ or $c_{\alpha,\beta,\ldots}$ to denote
a positive constant depending on the indicated parameters $\alpha,\beta,\ldots.$
We always let $\mathbb{N}:=\{1,2,\ldots\}$.  The \emph{symbols} $\mathbb{R}$ and $\mathbb{C}$
respectively denote the \emph{real} and the \emph{complex} numbers.
We identify $\mathbb{R}^2$ with $\mathbb{C}$ by the bijection
$(x_1,x_2)\leftrightarrow x_1+i x_2$, where $i^2=-1$.  For $z\in\mathbb{C}$, $\overline z$,
$\operatorname{Re} z$, $\operatorname{Im} z$, and $|z|$ respectively denote its \emph{complex
conjugate}, \emph{real part}, \emph{imaginary part}, and \emph{modulus}.
The \emph{unit circle} $\mathbb{S}^1$ is defined by setting $\mathbb{S}^1:=\{z\in\mathbb{C}:|z|=1\}$.
For a scalar or complex-valued function $v$ on an open set in $\mathbb{R}^2$, we
write
\begin{align*}
\partial_jv:=\frac{\partial v}{\partial x_j},\quad
\nabla v:=(\partial_1v,\partial_2v),\quad\text{and}\quad
\Delta v:=\partial_1^2v+\partial_2^2v.
\end{align*}
For a real vector field $F:=(F_1,F_2)$, its divergence and its curl  are defined respectively as
\begin{align*}
\operatorname{div} F:=\partial_1F_1+\partial_2F_2
\quad\text{and}\quad
\operatorname{curl} F:=\partial_1F_2-\partial_2F_1.
\end{align*}
For a map $V=(V_1,V_2):\Omega\to\mathbb{R}^2$, its \emph{Jacobian matrix} is defined as
$\nabla V:=(\partial_jV_k)_{1\le j, k\le 2}$ and its \emph{Frobenius norm} is defined as
$|\nabla V|:=(\sum_{j,k=1}^2|\partial_jV_k|^2)^{1/2}.$
In particular, for a real-valued function $f$ on an open set $\Omega$ in $\mathbb{R}^2$,
its \emph{Hessian matrix} is defined as $\nabla^2f=(\partial_j\partial_kf)_{j,k=1}^2$ with
$|\nabla^2f|:=(\sum_{j,k=1}^2|\partial_j\partial_kf|^2)^{1/2}$.  If $M=(M_{jk})$ is a
real $2\times2$ matrix, then $M^T:=(M_{kj})$ denotes its \emph{transpose} and
$\det M:=M_{11}M_{22}-M_{12}M_{21}$ denotes its \emph{determinant}.  If
$x:=(x_1,x_2),y:=(y_1,y_2)\in\mathbb{R}^2$, then $x\otimes y$ denotes the matrix with entries
$\{(x\otimes y)_{jk}:=x_jy_k\}_{j,k=1}^2$. Moreover, the $2\times2$ \emph{identity matrix} is denoted
by ${\mathrm I}_2$.
For any $R\in(0,\infty)$, define
\begin{align*}
 B_R:=\{x\in\mathbb{R}^2:|x|<R\}\quad\text{and}\quad
\Omega_R:=\mathbb{R}^2\setminus\overline{B}_R,
\end{align*}
where, for a measurable set $E\subset\mathbb{R}^2$, $\overline{E}$ denotes the \emph{closure} of
$E$ in $\mathbb{R}^2$. We also use ${\bf0}$ to denote the \emph{origin} of $\mathbb{R}^2$.
Let $\Omega\subset\mathbb{R}^2$ be an open set. For any given $k\in\mathbb{N}$ and
$q\in[1,\infty)$, $W^{k,q}(\Omega)$ denotes the \emph{Sobolev space} on $\Omega$. In particular,
when $k=1$ and $q=2$, let $H^1(\Omega):=W^{1,2}(\Omega)$ and denote by $H_0^1(\Omega)$ the
closure of $C_{\rm c}^\infty(\Omega)$ in $H^1(\Omega)$, where $C_{\rm c}^\infty(\Omega)$
denotes the set of all infinitely differentiable functions on $\Omega$ with compact
support contained in $\Omega$. The space $C^\infty(\overline\Omega)$ consists of the restrictions to
$\overline\Omega$ of all  functions smooth on an open neighbourhood of
$\overline\Omega$.  For any given $\alpha\in(0,1)$,
$C^{1,\alpha}(\overline\Omega)$ denotes the set of all functions whose first
derivatives are H\"older continuous with exponent $\alpha$ on
$\overline\Omega$. Furthermore, a subscript ``$\mathrm{loc}$'' means that the stated property
holds on every compact subset of the open set.
If $F\in L^1_{\mathrm{loc}}(\Omega;\mathbb{R}^2)$, the statement
$\operatorname{div} F=0$ in the sense of distributions means that, for
any $\varphi\in C_{\rm c}^\infty(\Omega)$,
\begin{align*}
\int_\Omega F\cdot\nabla\varphi\,dx=0.
\end{align*}

\section{Proof of Theorem \ref{thm:main}}\label{s3}

To show Theorem \ref{thm:main}, we first establish the following maximum modulus estimate
for the solution to \eqref{eq:GL} under the assumption \eqref{eq:modulus-limit}.

\begin{proposition}\label{lem:maxmod}
Assume that the solution $u\in C^\infty(\mathbb{R}^2;\mathbb{C})$ to \eqref{eq:GL}
has the asymptotic property \eqref{eq:modulus-limit}. Then
\begin{equation}\label{eq:maxmod}
|u(x)|\le1\quad\hbox{for all}\ x\in\mathbb{R}^2.
\end{equation}
\end{proposition}

\begin{proof}
Let $v:=|u|^2$.  Since the equation \eqref{eq:GL} is equivalent to $\Delta u=u(|u|^2-1)$, it follows that
\begin{equation}\label{eq:Deltav}
\Delta v=2|\nabla u|^2+2\mathrm{Re}(\overline{u}\Delta u)=2|\nabla u|^2+2v(v-1).
\end{equation}
By \eqref{eq:modulus-limit}, we find that $v(x)\to1$ as $|x|\to\infty$, which, combined with the assumption
$u\in C^\infty(\mathbb{R}^2;\mathbb{C})$, further implies that $v$ is bounded on $\mathbb{R}^2$.  If
$M:=\sup_{x\in\mathbb{R}^2}v(x)>1$, choose $R$ so large that
$v(x)<(M+1)/2<M$ whenever $|x|\ge R$.  Then there exists a point $x_0\in B_R$ such that $v(x_0)=M$,
which implies that $\Delta v(x_0)\le0$. Moreover, from \eqref{eq:Deltav}, we infer that
$$\Delta v(x_0)\ge2M(M-1)>0,$$
which contradicts $\Delta v(x_0)\le0$. Therefore, $\sup_{x\in\mathbb{R}^2}v(x)\le1$ and hence
\eqref{eq:maxmod} holds. This finishes the proof of Proposition \ref{lem:maxmod}.
\end{proof}

Furthermore, we also need the following $L^2$-integrability of the function
$h:=1-|u|$ on an exterior domain $\Omega_{R_0}$, whose proof will be given later in Section \ref{s8}.

\begin{proposition}\label{prop:h-L2}
Let $u\in C^\infty(\mathbb{R}^2;\mathbb{C})$ be a solution
to \eqref{eq:GL}  satisfying \eqref{eq:modulus-limit}, and let
$h:=1-|u|$. Then there exists $R_0\in(0,\infty)$ such that
\begin{equation}\label{eq:h-L2}
 h\in L^2(\Omega_{R_0};\mathbb{R}).
\end{equation}
\end{proposition}

Now, we prove Theorem~\ref{thm:main} by using Propositions~\ref{lem:maxmod} and ~\ref{prop:h-L2}.

\begin{proof}[Proof of Theorem~\ref{thm:main}]
Let $\rho:=|u|$. By Proposition~\ref{lem:maxmod}, we find that $0\le\rho\le1$.
Moreover, from Proposition~\ref{prop:h-L2}, we infer that $h=1-\rho\in L^2(\Omega_{R_0})$,
where $R_0\in(0,\infty)$ is as in Proposition~\ref{prop:h-L2}. Therefore,
$$
1-|u|^2=1-\rho^2=(1+\rho)(1-\rho)=(1+\rho)h\in L^2(\Omega_{R_0}).
$$
Furthermore, by $|\overline{B}_{R_0}|<\infty$ and the proved conclusion that $|u|\le1$ on $\mathbb{R}^2$ in
Proposition~\ref{lem:maxmod}, we conclude that $1-|u|^2\in L^2(\overline{B}_{R_0})$.
Thus, $1-|u|^2\in L^2(\mathbb{R}^2)$ and hence \eqref{e1.2} holds. This finishes the proof of
Theorem~\ref{thm:main}.
\end{proof}

\section{Some asymptotic properties}\label{s4}
To show Proposition \ref{prop:h-L2}, in this section we first establish
some asymptotic properties of  the solution to \eqref{eq:GL}.

\begin{lemma}\label{lem:grad-decay}
Let the solution $u\in C^\infty(\mathbb{R}^2;\mathbb{C})$ to \eqref{eq:GL}
satisfy \eqref{eq:modulus-limit}. Then
\begin{equation}\label{eq:grad-decay}
|\nabla u(x)|\rightarrow0\quad\hbox{as}\quad|x|\to\infty.
\end{equation}
\end{lemma}

To prove Lemma \ref{lem:grad-decay}, we need the following conclusion.

\begin{lemma}\label{lem:interior}
Let $q\in(2,\infty)$ and $R\in(0,\infty)$. If $v\in W^{2,q}(B_{2R})$ satisfies
$-\Delta v=f$ in $B_{2R}$, then
\begin{align}\label{eq:W2q}
\left\|v\right\|_{W^{2,q}(B_R)}\le C_{q,R}\left[\left\|v\right\|_{L^q(B_{2R})}
+\left\|f\right\|_{L^q(B_{2R})}\right].
\end{align}
Consequently, if $\{v_n\}_{n\in\mathbb{N}}\subset W^{2,q}_{\mathrm{loc}}(\mathbb{R}^2)$ and, for
any given $R\in(0,\infty)$,
\begin{align}\label{e2.1}
\sup_{n\in\mathbb{N}}\left[\left\|v_n\right\|_{L^q(B_{2R})}
+\left\|\Delta v_n\right\|_{L^q(B_{2R})}\right]<\infty,
\end{align}
then there exists a subsequence of $\{v_n\}_{n\in\mathbb{N}}$ converging in
$C^1_{\mathrm{loc}}(\mathbb{R}^2)$.
\end{lemma}

\begin{proof}
The estimate \eqref{eq:W2q} is the constant-coefficient special case of
\cite[Chapter~9, Theorem~9.11]{gt01}. Moreover, for any given $R\in(0,\infty)$, applying
\cite[Theorem~7.26(ii)]{gt01}, we find that, for
any fixed $\alpha\in(0,1-2/q)$, $W^{2,q}(B_R)$ is compactly imbedded into $C^{1,\alpha}(\overline{B}_R)$.
From this and the assumption \eqref{e2.1}, we deduce that there exists a subsequence
of $\{v_n\}_{n\in\mathbb{N}}$ converging in  $C^1(\overline{B}_R)$.  Applying this conclusion on
the ball $B_m$ with $m\in\mathbb{N}$ and taking a diagonal subsequence, we further
find that there exists a subsequence of $\{v_n\}_{n\in\mathbb{N}}$ converging in
$C^1_{\mathrm{loc}}(\mathbb{R}^2)$. This finishes the proof of Lemma \ref{lem:interior}.
\end{proof}

Now, we show Lemma \ref{lem:grad-decay} by using Proposition \ref{lem:maxmod}
and Lemma \ref{lem:interior}.

\begin{proof}[Proof of Lemma \ref{lem:grad-decay}]
We employ the method of contradiction to prove the present  lemma. Assume that \eqref{eq:grad-decay} fails.
Then there exist a constant $\varepsilon_0\in(0,\infty)$ and a sequence $\{x_n\}_{n\in\mathbb{N}}$
of points in $\mathbb{R}^2$ such that $|x_n|\to\infty$ as $n\to\infty$, but
$|\nabla u(x_n)|\ge\varepsilon_0$ for any $n\in\mathbb{N}$. For any given
$n\in\mathbb{N}$, define $u_n(y):=u(x_n+y)$ for any $y\in\mathbb{R}^2$.
By Proposition \ref{lem:maxmod}, we find that, for any $n\in\mathbb{N}$, $|u_n|\le1$,
\begin{equation}\label{e3.1}
-\Delta u_n=f_n:=u_n\left(1-|u_n|^2\right),
 \quad \text{and}\quad \|f_n\|_{L^\infty(\mathbb{R}^2;\mathbb{C})}\le1.
\end{equation}
Fix $q\in(2,\infty)$. Then, from \eqref{e3.1}, Lemma~\ref{lem:interior}, and
the weak compactness of  $W^{2,q}(\mathbb{R}^2;\mathbb{C})$, we deduce that there exist a subsequence
of $\{u_n\}_{n\in\mathbb{N}}$, still denoted by $\{u_n\}_{n\in\mathbb{N}}$, and $U\in W^{2,q}_{\mathrm{loc}}
(\mathbb{R}^2;\mathbb{C})\cap C^1_{\mathrm{loc}}(\mathbb{R}^2;\mathbb{C})$
such that $u_n\to U$ in $C^1_{\mathrm{loc}}(\mathbb{R}^2;\mathbb{C})$
and weakly in $W^{2,q}_{\mathrm{loc}}(\mathbb{R}^2;\mathbb{C})$ as $n\to\infty$.

Moreover, by the fact that $|x_n|\to\infty$ as $n\to\infty$, we conclude that, for any given compact set
$K\subset\mathbb{R}^2$, $\inf_{y\in K}|x_n+y|\to\infty$ as $n\to\infty$. This, combined with
\eqref{eq:modulus-limit}, implies that $|u_n|\to1$ uniformly on $K$ as $n\to\infty$, and
hence $|U|\equiv1$ on $K$. Therefore, $|U|\equiv1$ on $\mathbb{R}^2$. Meanwhile,
from \eqref{eq:modulus-limit} and the proved fact that $|u_n|\le1$ for any $n\in\mathbb{N}$,
it follows that $f_n\to0$ uniformly on $K$ as $n\to\infty$. Therefore, for any
$\varphi\in C_{\rm c}^\infty(\mathbb{R}^2;\mathbb{C})$,
$$
\int_{\mathbb{R}^2}U\Delta\varphi\,dy=\lim_{n\to\infty}\int_{\mathbb{R}^2}u_n\Delta\varphi\,dy
=-\lim_{n\to\infty}\int_{\mathbb{R}^2}f_n\varphi\,dy=0.
$$
Thus, $\Delta U=0$ in the sense of distributions.  Since $U\in W^{2,q}_{\mathrm{loc}}(\mathbb{R}^2;\mathbb{C})$,
it follows that
$$
 \frac12\Delta|U|^2=|\nabla U|^2+\mathrm{Re}(\overline{U}\Delta U)=|\nabla U|^2.
$$
Furthermore, by $|U|\equiv1$, we find that $|\nabla U|=0$ almost everywhere  and hence
everywhere by continuity.  On the other hand, from the proved fact that $u_n\to U$ in
$C^1_{\mathrm{loc}}(\mathbb{R}^2;\mathbb{C})$ as $n\to\infty$, we deduce that
$|\nabla U({\bf0})|=\lim_{n\to\infty}|\nabla u(x_n)|\ge\varepsilon_0$,
which contradicts $|\nabla U({\bf0})|=0$. Thus, \eqref{eq:grad-decay} holds. This finishes
the proof of Lemma \ref{lem:grad-decay}.
\end{proof}

Let $u\in C^\infty(\mathbb{R}^2;\mathbb{C})$ satisfy \eqref{eq:modulus-limit}.
Choose $R_1\in(0,\infty)$ sufficiently large such that, for any $x\in\Omega_{R_1}$,
\begin{equation}\label{eq:rho-lower}
|u(x)|\ge\frac12.
\end{equation}
On $\Omega_{R_1}$, for $j\in\{1,2\}$, let
\begin{equation}\label{eq:rho-p-def}
\rho:=|u|,\quad p_j:=\frac{\mathrm{Im}(\overline u\,\partial_ju)}{\rho^2},\quad
\text{and} \quad p:=(p_1,p_2).
\end{equation}

\begin{lemma}\label{lem:sigma}
Let the solution $u\in C^\infty(\mathbb{R}^2;\mathbb{C})$ to \eqref{eq:GL} satisfy \eqref{eq:modulus-limit},
and let $R_1$ be as in \eqref{eq:rho-p-def}. Assume that $\rho$ and $p$ are as in \eqref{eq:rho-p-def} and
$\sigma:=\frac{u}\rho$ on $\Omega_{R_1}.$
Then $\sigma\in C^\infty(\Omega_{R_1};\mathbb{S}^1)$ and, on $\Omega_{R_1}$,
for $j\in\{1,2\}$,
\begin{equation}\label{eq:sigma-derivative}
\overline\sigma\,\partial_j\sigma=i p_j\quad\text{and}\quad
 \partial_j\sigma=i p_j\sigma.
\end{equation}
Consequently, on $\Omega_{R_1}$,
\begin{equation}\label{eq:sigma-Laplacian}
\Delta\sigma=\left(i\,\mathrm{div} p-|p|^2\right)\sigma\quad\text{and}\quad
\mathrm{curl} p=0.
\end{equation}
\end{lemma}

\begin{proof}
Since $|\sigma|^2=1$, it follows that, for $j\in\{1,2\}$,
$\mathrm{Re}(\overline\sigma\,\partial_j\sigma)=0$. By the definition that $u=\rho\sigma$,
we conclude that, for $j\in\{1,2\}$,
$$
\overline u\,\partial_ju=\rho\partial_j\rho+\rho^2\overline\sigma\,\partial_j\sigma,
$$
which, combined with \eqref{eq:rho-p-def} and the fact that the term $\rho\partial_j\rho$
is real, implies that, for $j\in\{1,2\}$, $\overline\sigma\,\partial_j\sigma=i p_j$.
This, together with the fact that $|\sigma|=1=\overline\sigma\sigma$, further yields
$\partial_j\sigma=i p_j\sigma$ for $j\in\{1,2\}$. Therefore, \eqref{eq:sigma-derivative} holds.

Furthermore, using the second identity in \eqref{eq:sigma-derivative}, we further obtain
$$
\Delta\sigma=\sum_{j=1}^2\partial_j(i p_j\sigma)
=i(\mathrm{div}p)\sigma-|p|^2\sigma\quad\text{on}\quad\Omega_{R_1}.
$$
Meanwhile,  by the fact that $\partial_1\partial_2\sigma=\partial_2\partial_1\sigma$
on $\Omega_{R_1}$ and \eqref{eq:sigma-derivative} again, we find that
$$
0=\partial_1\partial_2\sigma-\partial_2\partial_1\sigma
=i(\partial_1p_2-\partial_2p_1)\sigma,
$$
which, combined with $|\sigma|=1$, implies that $\mathrm{curl} p=0$ on $\Omega_{R_1}$.
Therefore, \eqref{eq:sigma-Laplacian} holds. This finishes the proof of Lemma \ref{lem:sigma}.
\end{proof}

\begin{lemma}\label{lem:amplitude-phase}
Let $R_1$ be as in \eqref{eq:rho-p-def} and $\rho$ and $p$ be as in \eqref{eq:rho-p-def}.
Then, on $\Omega_{R_1}$,
\begin{align}
-\Delta\rho+\rho|p|^2=\rho(1-\rho^2)\label{eq:amplitude}
\end{align}
and
\begin{align}
\mathrm{div}(\rho^2p)=0.\label{eq:phase-div}
\end{align}
Moreover,
\begin{equation}\label{eq:rho-p-limits}
\rho\to1,\quad \nabla\rho\to{\bf0},\quad\text{and}\quad p\to{\bf0}\quad\hbox{uniformly as}\quad|x|\to\infty.
\end{equation}
\end{lemma}

\begin{proof}
Using $u=\rho\sigma$, \eqref{eq:sigma-derivative}, and \eqref{eq:sigma-Laplacian}, we find that
$$
\Delta u=\sigma\bigl[\Delta\rho-\rho|p|^2
+i(2\nabla\rho\cdot p+\rho\mathrm{div}p)\bigr].
$$
Substituting this equation into \eqref{eq:GL}, multiplying by $\overline\sigma$,
and then comparing the real and the imaginary parts yield \eqref{eq:amplitude} and the identity
$2\nabla\rho\cdot p+\rho\mathrm{div}p=0$.  Multiplying the latter identity by
$\rho$ gives \eqref{eq:phase-div}.

The first limit in \eqref{eq:rho-p-limits} is precisely \eqref{eq:modulus-limit}.
Moreover, since $\rho>0$ on $\Omega_{R_1}$ and, for $j\in\{1,2\}$,
$$\partial_j\rho=\frac{u\cdot\partial_ju}{|u|}\quad\text{and}\quad
|p_j|\le\frac{|\partial_ju|}{\rho},
$$
it follows from Lemma~\ref{lem:grad-decay} and \eqref{eq:rho-lower} that the remaining
two uniform limits in \eqref{eq:rho-p-limits} hold. This finishes the proof of Lemma
\ref{lem:amplitude-phase}.
\end{proof}

\section{Translation Jacobi system and finite tail energy}\label{s5}

In this section, we study the translation Jacobi system for the derivatives of
the amplitude and the phase. Precisely, fix $\ell\in\{1,2\}$ and define the real-valued functions
\begin{equation}\label{eq:Jacobi-definitions}
h:=1-\rho,\quad a_\ell:=\partial_\ell h=-\partial_\ell\rho,\quad
b_\ell:=p_\ell,\quad\mathfrak s:=|p|^2,\quad
\text{and}\quad m:=\mathfrak s-1+3\rho^2.
\end{equation}
Based on \eqref{eq:rho-p-limits}, choose $R_1\in(0,\infty)$ sufficiently large such that
\begin{equation}\label{eq:tail-smallness}
\rho\ge\frac34,\quad|p|\le\frac14,
\end{equation}
and
\begin{equation}\label{eq:mass-Schur}
m-4\mathfrak s=-1+3\rho^2-3|p|^2\ge\frac12
\quad\text{on}\quad \Omega_{R_1}.
\end{equation}

\begin{lemma}\label{lem:Jacobi-system}
Let $R_1$ be as in \eqref{eq:mass-Schur}, $\ell\in\{1,2\}$, and $h$, $a_\ell$, $b_\ell$,
and $m$ be as in \eqref{eq:Jacobi-definitions}. Then, on $\Omega_{R_1}$,
\begin{align}
(-\Delta+m)a_\ell=2\rho p\cdot\nabla b_\ell\label{eq:J1}
\end{align}
and
\begin{align}
\mathrm{div}\left(\rho^2\nabla b_\ell-2\rho a_\ell p\right)=0.\label{eq:J2}
\end{align}
\end{lemma}
\begin{proof}
Let $\ell,j\in\{1,2\}$. Differentiating \eqref{eq:amplitude} with respect to $x_\ell$
and using the fact that $\operatorname{curl} p=0$, we find that
$\partial_\ell p_j=\partial_j p_\ell=\partial_j b_\ell$ and hence $\partial_\ell p=\nabla b_\ell$.
Consequently,
$$
-\Delta(\partial_\ell\rho)+\left(\mathfrak{s}-1+3\rho^2\right)\partial_\ell\rho
+2\rho\, p\cdot\nabla b_\ell=0.
$$
Substituting $a_\ell=-\partial_\ell\rho$ into the above equation yields \eqref{eq:J1}.
Similarly, differentiating \eqref{eq:phase-div} and using $\operatorname{curl} p=0$ again,
we conclude that
$$
\operatorname{div}\left(2\rho(\partial_\ell\rho)p+\rho^2\nabla b_\ell\right)=0,
$$
which is precisely \eqref{eq:J2}. This finishes the proof of Lemma \ref{lem:Jacobi-system}.
\end{proof}

\begin{lemma}\label{lem:Schur}
Let $\rho$ and $p$ be as in \eqref{eq:rho-p-def}, $m$ and $\mathfrak s$ be as in \eqref{eq:Jacobi-definitions},
and $R_1$ be as in \eqref{eq:mass-Schur}. Then, for any $A\in\mathbb{R}$ and $X,Y\in\mathbb{R}^2$,
\begin{align}\label{eq:Schur}
|X|^2+\rho^2\left|Y-\frac{2A}{\rho}p\right|^2+(m-4\mathfrak s)A^2
\ge\frac14\left(|X|^2+|Y|^2+A^2\right)\quad\text{on}\quad\Omega_{R_1}.
\end{align}
\end{lemma}

\begin{proof}
Fix $A\in\mathbb{R}$ and $X,Y\in\mathbb{R}^2$ and let $Y_0:=Y-(2A/\rho)p$.
Then, by \eqref{eq:tail-smallness}, we conclude that $|(2A/\rho)p|\le(2/3)|A|$ on
$\Omega_{R_1}$, and hence
$$
|Y|^2\le2|Y_0|^2+\frac89A^2,
$$
which further implies that
$$
|Y|^2+A^2\le2|Y_0|^2+\frac{17}{9}A^2\le4\left(\frac9{16}|Y_0|^2+\frac12A^2\right).
$$
From this, \eqref{eq:tail-smallness}, and \eqref{eq:mass-Schur}, we infer that
\eqref{eq:Schur} holds on $\Omega_{R_1}$. This finishes the proof of Lemma \ref{lem:Schur}.
\end{proof}

\begin{lemma}\label{lem:Caccioppoli}
Let $a_\ell$ and $b_\ell$ with $\ell\in\{1,2\}$ be as in \eqref{eq:Jacobi-definitions} and
$R_1$ be as in \eqref{eq:mass-Schur}. Then there exists a positive constant $C$ such that, for any
$\varphi\in C_{\rm c}^\infty(\Omega_{R_1};\mathbb{R})$,
\begin{equation}\label{eq:Caccioppoli}
\int_{\Omega_{R_1}}\varphi^2\left(|\nabla a_\ell|^2+a_\ell^2+|\nabla b_\ell|^2\right)\,dx
\le C\int_{\Omega_{R_1}}(a_\ell^2+b_\ell^2)|\nabla\varphi|^2\,dx.
\end{equation}
\end{lemma}

\begin{proof}
Let $\ell\in\{1,2\}$ and $\varphi\in C_{\rm c}^\infty(\Omega_{R_1};\mathbb{R})$.
Note that $a_\ell$ and $b_\ell$ satisfy the equation \eqref{eq:J1}. Taking
$\varphi^2a_\ell$ as a test function, we then obtain
\begin{align}
\int_{\Omega_{R_1}}\varphi^2|\nabla a_\ell|^2\,dx+\int_{\Omega_{R_1}} m\varphi^2a_\ell^2\,dx
-2\int_{\Omega_{R_1}}\rho\varphi^2a_\ell p\cdot\nabla b_\ell\,dx
+2\int_{\Omega_{R_1}}\varphi a_\ell\nabla a_\ell\cdot\nabla\varphi\,dx=0.
\label{eq:test-a}
\end{align}
Moreover, using the fact that $a_\ell$ and $b_\ell$ satisfy \eqref{eq:J2} and
taking $\varphi^2b_\ell$ as a test function in \eqref{eq:J2}, we find that
\begin{align}
&\int_{\Omega_{R_1}}\rho^2\varphi^2|\nabla b_\ell|^2\,dx-2\int_{\Omega_{R_1}}
\rho\varphi^2a_\ell p\cdot\nabla b_\ell\,dx\notag\\
 &\quad+2\int_{\Omega_{R_1}}\rho^2\varphi b_\ell\nabla b_\ell\cdot\nabla\varphi\,dx
 -4\int_{\Omega_{R_1}}\rho\varphi a_\ell b_\ell p\cdot\nabla\varphi\,dx=0.
\label{eq:test-b}
\end{align}
Then, from \eqref{eq:test-a} and \eqref{eq:test-b}, it follows that
\begin{align}\label{e4.1}
&\int_{\Omega_{R_1}}\varphi^2\left(|\nabla a_\ell|^2+ma_\ell^2+\rho^2|\nabla b_\ell|^2-
4\rho a_\ell p\cdot\nabla b_\ell\right)\,dx\notag\\
&\quad=-2\int_{\Omega_{R_1}}\varphi a_\ell\nabla a_\ell\cdot\nabla\varphi\,dx
-2\int_{\Omega_{R_1}}\rho^2\varphi b_\ell\nabla b_\ell\cdot\nabla\varphi\,dx\notag\\
&\qquad+4\int_{\Omega_{R_1}}\rho\varphi a_\ell b_\ell p\cdot\nabla\varphi\,dx.
\end{align}
By \eqref{eq:Schur}, we find that, on $\Omega_{R_1}$,
\begin{align}\label{e4.2}
|\nabla a_\ell|^2+ma_\ell^2+\rho^2|\nabla b_\ell|^2-4\rho a_\ell p\cdot\nabla b_\ell
&=|\nabla a_\ell|^2+\rho^2\left|\nabla b_\ell-\frac{2a_\ell}{\rho}p\right|^2
+(m-4\mathfrak s)a_\ell^2\notag\\
&\ge\frac14\left(|\nabla a_\ell|^2+|\nabla b_\ell|^2+a_\ell^2\right).
\end{align}
Moreover, from Young's inequality, we infer that,  on $\Omega_{R_1}$, for any $\varepsilon\in(0,\infty)$,
\begin{align}\label{e4.3}
2|\varphi a_\ell\nabla a_\ell\cdot\nabla\varphi|
&\le\varepsilon\varphi^2|\nabla a_\ell|^2
 +C_\varepsilon a_\ell^2|\nabla\varphi|^2,\notag\\
2\rho^2|\varphi b_\ell\nabla b_\ell\cdot\nabla\varphi|
&\le\varepsilon\varphi^2|\nabla b_\ell|^2
 +C_\varepsilon b_\ell^2|\nabla\varphi|^2,\ \ \text{and}\notag\\
4\rho|p|\,|\varphi a_\ell b_\ell|\,|\nabla\varphi|
&\le\varepsilon\varphi^2a_\ell^2
 +C_\varepsilon b_\ell^2|\nabla\varphi|^2.
\end{align}
Letting $\varepsilon:=1/24$ in \eqref{e4.3} and using \eqref{e4.1}, \eqref{e4.2},
and \eqref{e4.3}, we conclude that \eqref{eq:Caccioppoli} holds.
This finishes the proof of Lemma \ref{lem:Caccioppoli}.
\end{proof}

\begin{proposition}\label{prop:tail-energy}
Let $\ell\in\{1,2\}$, $a_\ell$ and $b_\ell$ be as in \eqref{eq:Jacobi-definitions}, and
$R_1$ be as in \eqref{eq:mass-Schur}. Then there exists $R_2\in(R_1,\infty)$ such that
\begin{equation}\label{eq:tail-energy}
\int_{\Omega_{R_2}}\left(|\nabla a_\ell|^2+a_\ell^2+|\nabla b_\ell|^2\right)\,dx<\infty.
\end{equation}
\end{proposition}

\begin{proof}
Fix $R_2\in(R_1+2,\infty)$.  Let $\varphi_1\in C^\infty(\mathbb{R}^2;\mathbb{R})$ be
radial and satisfy $0\le\varphi_1\le1$, $\varphi_1\equiv0$ on $B_{R_1+1}$, and
$\varphi_1\equiv1$ on $\Omega_{R_2}$.  For any given $L\in(2R_2,\infty)$, take the radial
function $\varphi_2\in C_{\rm c}^\infty(\mathbb{R}^2;\mathbb{R})$ be such that $\varphi_2\equiv1$ on $B_L$,
$\varphi_2\equiv0$ on $\Omega_{2L}$, and  $|\nabla\varphi_2|\le C/L$. Let $\varphi_L:=\varphi_1\varphi_2$.
Therefore, for any given $L\in(2R_2,\infty)$,
\begin{align}\label{e4.4}
\int_{\Omega_{R_2}}\left(a_\ell^2+b_\ell^2\right)|\nabla\varphi_L|^2\,dx&=
\int_{B_{2L}\setminus B_L}\left(a_\ell^2+b_\ell^2\right)|\nabla\varphi_L|^2\,dx\notag\\
&\le\frac{C}{L^2}|B_{2L}\setminus B_L|\sup_{B_{2L}\setminus B_L}\left(a_\ell^2+b_\ell^2\right)
\le C\sup_{B_{2L}\setminus B_L}\left(a_\ell^2+b_\ell^2\right).
\end{align}
Moreover, it is easy to find that, for any $x\in\mathbb{R}^2$,
$$[\varphi_L(x)]^2\left(|\nabla a_\ell(x)|^2+[a_\ell(x)]^2+|\nabla b_\ell(x)|^2\right)
\to\left(|\nabla a_\ell(x)|^2+[a_\ell(x)]^2+|\nabla b_\ell(x)|^2\right)
$$
as $L\to\infty,$
which, together with Fatou's lemma, implies that
\begin{align}\label{e4.4x}
\int_{\Omega_{R_2}}\left(|\nabla a_\ell|^2+a_\ell^2+|\nabla b_\ell|^2\right)\,dx
\le\liminf_{L\to\infty}\int_{\Omega_{R_2}}\varphi_L^2\left(|\nabla a_\ell|^2+a_\ell^2+|\nabla b_\ell|^2\right)\,dx.
\end{align}
By \eqref{eq:rho-p-limits}, we find that both $a_\ell(x)=-\partial_\ell\rho(x)$ and
$b_\ell(x)=p_\ell(x)$ tend to $0$ uniformly as $|x|\to\infty$, which
implies that $\sup_{B_{2L}\setminus B_L}(a_\ell^2+b_\ell^2)\to0$ as $L\to\infty$.
Using this, \eqref{e4.4}, \eqref{eq:Caccioppoli}, and \eqref{e4.4x}, we conclude that
\eqref{eq:tail-energy} holds. This finishes the proof of Proposition \ref{prop:tail-energy}.
\end{proof}

\section{Kelvin inversion and quantitative decay at the inversion point}\label{s6}

In this section, by using the  Kelvin inversion argument, we transform the exterior problem in Section
\ref{s5} to a singular system near the origin. Precisely, let $R_2\in(0,\infty)$ be as in
Proposition \ref{prop:tail-energy} and $r_0:=R_2^{-1}$.
For any $y\in\mathbb{R}^2\setminus\{{\bf0}\}$, define its \emph{Kelvin inversion},
\emph{unit radial vector}, and \emph{reflection matrix}, respectively, by setting
\begin{equation}\label{eq:Kelvin-def}
I(y):=(I_1(y),I_2(y)):=\frac{y}{|y|^2},\quad \widehat y:=\frac{y}{|y|},
\quad\text{and}\quad Q(y):=\mathrm{I}_2-2\widehat y\otimes\widehat y.
\end{equation}
Then, for any $y\in\mathbb{R}^2\setminus\{{\bf0}\}$,
\begin{equation}\label{eq:Kelvin-differential}
\nabla I(y)=|y|^{-2}Q(y),\quad Q(y)^T=Q(y),\quad Q(y)^2=\mathrm{I}_2,\quad\text{and}\quad |\det \nabla I(y)|=|y|^{-4}.
\end{equation}
Let $\ell\in\{1,2\}$. For any $y\in\mathbb{R}^2$ satisfying $0<|y|<r_0$, define
\begin{equation}\label{eq:Kelvin-unknowns}
\begin{aligned}
A_\ell(y)&:=a_\ell(I(y)),&B_\ell(y)&:=b_\ell(I(y)),&P(y)&:=Q(y)p(I(y)),\\
\rho_K(y)&:=\rho(I(y)),\ \ \text{and}\ \ & m_K(y)&:=m(I(y)).&&
\end{aligned}
\end{equation}
By \eqref{eq:rho-p-limits}, we find that these quantities have the continuous extensions at $y={\bf0}$. Precisely,
\begin{equation}\label{eq:Kelvin-values-zero}
A_\ell({\bf0})=B_\ell({\bf0})=0, \quad P({\bf0})={\bf0}, \quad \rho_K({\bf0})=1,\quad \text{and}\quad m_K({\bf0})=2.
\end{equation}
Indeed, although $Q(y)$ has no limit as $y\to{\bf0}$, the definition $P({\bf0})={\bf0}$ is
continuous because $p(I(y))\to{\bf0}$ uniformly as $y\to{\bf0}$.

\begin{lemma}\label{lem:Kelvin-system}
Let $R_2\in(0,\infty)$ be as in Proposition \ref{prop:tail-energy}, $r_0:=R_2^{-1}$, and $\ell\in\{1,2\}$.
Assume $A_\ell,\ B_\ell,\ P,\ \rho_K,$ and $m_K$ are as in \eqref{eq:Kelvin-unknowns}.
Then, on $B_{r_0}\setminus\{{\bf0}\}$,
\begin{align}
-\Delta A_\ell+|\cdot|^{-4}m_KA_\ell=2|\cdot|^{-2}\rho_KP\cdot\nabla B_\ell
\label{eq:KJ1}
\end{align}
and
\begin{align}
\mathrm{div}\left(\rho_K^2\nabla B_\ell-2|\cdot|^{-2}\rho_KA_\ell P\right)=0.
 \label{eq:KJ2}
\end{align}
Furthermore,
\begin{equation}\label{eq:Kelvin-energy}
\int_{B_{r_0}}\left(|\nabla A_\ell(y)|^2+|y|^{-4}[A_\ell(y)]^2
+|\nabla B_\ell(y)|^2\right)\,dy<\infty.
\end{equation}
\end{lemma}

\begin{proof}
Let $y\in B_{r_0}\setminus\{{\bf0}\}$ and $r:=|y|$. From \eqref{eq:Kelvin-differential}, it follows that
$\partial_jI_k(y)=r^{-2}Q_{kj}(y)$ for any $j,k\in\{1,2\}$.
Moreover, it is easy to find that each component of $I$ is harmonic on
$\mathbb{R}^2\setminus\{{\bf0}\}$. Therefore, if $F:=f\circ I$ with $f\in C^{\infty}
(\mathbb{R}^2;\mathbb{R})$, then the chain rule yields
\begin{equation}\label{eq:Kelvin-scalar}
\nabla_yF=r^{-2}Q(\nabla_xf)\circ I\quad\text{and}\quad
\Delta_yF=r^{-4}(\Delta_xf)\circ I.
\end{equation}
Thus, $(\nabla_xf)\circ I=r^2Q\nabla_yF$. From this,
\eqref{eq:J1}, and the fact that $p\circ I=QP$, we deduce that
\eqref{eq:KJ1} holds.

Let $F_x:=\rho^2\nabla b_\ell-2\rho a_\ell p$,
$\varphi\in C_{\rm c}^\infty(B_{r_0}\setminus\{{\bf0}\})$, and for any
$x\in B_{r_0}\setminus\{{\bf0}\}$, $\psi(x):=\varphi(I(x))$.
By $\mathrm{div}_xF_x=0$ in the sense of distributions, we find that
\begin{equation}\label{e5.1}
\int_{\Omega_{R_1}}F_x(x)\cdot\nabla_x\psi(x)\,dx=0.
\end{equation}
Make the change of variables $x=I(y)$. Then
$\nabla I(x)=r^2Q(y)$ and $dx=r^{-4}dy$. From this and \eqref{e5.1},
we infer that
$$
\int_{B_{r_0}\setminus\{{\bf0}\}}r^{-2}Q(y)F_x(I(y))\cdot\nabla_y\varphi(y)\,dy=0.
$$
This, combined with \eqref{eq:Kelvin-scalar}, implies that \eqref{eq:KJ2}
holds.

Furthermore, by \eqref{eq:Kelvin-scalar} and the Jacobian identity, we conclude
that
$$
\int_{B_{r_0}}|\nabla(f\circ I)|^2\,dy=\int_{\Omega_{R_2}}|\nabla f|^2\,dx
$$
and
$$
\int_{B_{r_0}}|y|^{-4}|a_\ell(I(y))|^2\,dy=\int_{\Omega_{R_2}}|a_\ell(x)|^2\,dx.
$$
Applying these identities to $a_\ell$ and $b_\ell$ and
using \eqref{eq:tail-energy}, we find that \eqref{eq:Kelvin-energy} holds.
This finishes the proof of Lemma \ref{lem:Kelvin-system}.
\end{proof}

\begin{lemma}\label{lem:point-removal}
Let $R_2\in(0,\infty)$ be as in Proposition \ref{prop:tail-energy}, $r_0:=R_2^{-1}$, and
$\ell\in\{1,2\}$. Assume $A_\ell,\ B_\ell,\ P,\ \rho_K,$ and $m_K$ are as in \eqref{eq:Kelvin-unknowns}. Let
\begin{equation*}
\mathcal F_\ell:=\rho_K^2\nabla B_\ell-2|\cdot|^{-2}\rho_KA_\ell P.
\end{equation*}
Then $A_\ell,B_\ell\in H^1(B_{r_0};\mathbb{R})$, $\mathcal F_\ell\in L^2(B_{r_0};\mathbb{R}^2)$, and
$\mathrm{div}\mathcal F_\ell=0$ in the sense of distributions on $B_{r_0}$.
Moreover, for any $\phi\in C_{\rm c}^\infty(B_{r_0};\mathbb{R})$,
\begin{align}\label{eq:mass-test}
&\int_{B_{r_0}}\nabla A_\ell\cdot\nabla\left(\phi^2A_\ell\right)\,dy
+\int_{B_{r_0}}|\cdot|^{-4}m_K\phi^2A_\ell^2\,dy=2\int_{B_{r_0}}|\cdot|^{-2}\rho_K\phi^2A_\ell
P\cdot\nabla B_\ell\,dy.
\end{align}
\end{lemma}

\begin{proof}
From \eqref{eq:Kelvin-energy}, we deduce that $|\nabla A_\ell|,|\nabla B_\ell|
\in L^2(B_{r_0};\mathbb{R})$. Moreover, $|\cdot|^{-4}\ge r_0^{-4}$ on $B_{r_0}$,
which, together with \eqref{eq:Kelvin-energy}, implies that $A_\ell\in
L^2(B_{r_0};\mathbb{R})$. Meanwhile, it is easy to find that the function $B_\ell=b_\ell\circ I$
is bounded because $p$ is bounded on
the closed set $\{x\in\mathbb{R}^2: |x|\ge R_1\}$, hence
$B_\ell\in L^2(B_{r_0};\mathbb{R})$. Thus, $A_\ell,
B_\ell\in H^1(B_{r_0};\mathbb{R})$.

Furthermore, since $\rho_K$ and $P$ are continuous on
$B_{r_0}$, it follows that $\rho_K$ and $P$ are bounded on $B_{r_0}$,
which further implies that
$$
\int_{B_{r_0}}\left||\cdot|^{-2}\rho_KA_\ell P\right|^2\,dy
\le C\|P\|_{L^\infty(B_{r_0};\mathbb{R}^2)}^2\int_{B_{r_0}}|\cdot|^{-4}|A_\ell|^2\,dy<\infty.
$$
Thus, $\mathcal F_\ell\in L^2(B_{r_0};\mathbb{R}^2)$.

Now, we show $\mathrm{div}\mathcal F_\ell=0$ in the sense of distributions on $B_{r_0}$.
Let $\varphi\in C_{\rm c}^\infty(B_{r_0};\mathbb{R})$. For any given
$\varepsilon\in(0,1)$, choose a radial function $\zeta_\varepsilon\in C^\infty_{\rm c}(\mathbb{R}^2;\mathbb{R})$
such that $0\le\zeta_\varepsilon\le1$, $\zeta_\varepsilon\equiv0$ on $B_\varepsilon$,
$\zeta_\varepsilon=1$ outside $B_{2\varepsilon}$, and $|\nabla\zeta_\varepsilon|\le C/\varepsilon$. Then
$\|\nabla\zeta_\varepsilon\|_{L^2(\mathbb{R}^2;\mathbb{R}^2)}\le C$ uniformly in
$\varepsilon$. Taking $\zeta_\varepsilon\varphi$ as a test function in
\eqref{eq:KJ2}, we then have
\begin{align}\label{e5.2}
\int_{B_{r_0}}\varphi\mathcal F_\ell\cdot\nabla\zeta_\varepsilon\,dx
+\int_{B_{r_0}}\zeta_\varepsilon\mathcal F_\ell\cdot\nabla\varphi\,dx=0.
\end{align}
By the proved facts that $\mathcal F_\ell\in L^2(B_{r_0};\mathbb{R}^2)$
and $\|\nabla\zeta_\varepsilon\|_{L^2(\mathbb{R}^2;\mathbb{R}^2)}\le C$ uniformly in
$\varepsilon$, we conclude that
\begin{align}\label{e5.3}
\left|\int_{B_{r_0}}\varphi\mathcal F_\ell\cdot\nabla\zeta_\varepsilon\,dx\right|
\le\|\varphi\|_{L^\infty(B_{r_0};\mathbb{R})}
\|\mathcal F_\ell\|_{L^2(B_{2\varepsilon};\mathbb{R}^2)}
\|\nabla\zeta_\varepsilon\|_{L^2(\mathbb{R}^2;\mathbb{R}^2)}\to0
\end{align}
as $\varepsilon\to0$.
Meanwhile, from the dominated convergence theorem, it follows that
\begin{align}\label{e5.4}
\int_{B_{r_0}}\zeta_\varepsilon\mathcal F_\ell\cdot\nabla\varphi\,dx\to0
\end{align}
as $\varepsilon\to0$. Therefore, by \eqref{e5.2}, \eqref{e5.3}, and \eqref{e5.4},
we conclude that
\begin{align*}
\int_{B_{r_0}}\mathcal F_\ell\cdot\nabla\varphi\,dx=0.
\end{align*}
That is, $\mathrm{div}\mathcal F_\ell=0$ in the sense of distributions
on $B_{r_0}$.

Finally, we show \eqref{eq:mass-test}. Let $\phi\in C_{\rm c}^\infty(B_{r_0};\mathbb{R})$ and, for any given
$\varepsilon$, $\zeta_\varepsilon$ be as in \eqref{e5.2}. Using
$\zeta_\varepsilon^2\phi^2A_\ell$ as a test function in \eqref{eq:KJ1}, we find that
\begin{align}\label{e5.5}
&\int_{B_{r_0}}\nabla A_\ell\cdot\nabla\left(\phi^2A_\ell\right)\,dy
+2\int_{B_{r_0}}\zeta_\varepsilon\phi^2A_\ell\nabla A_\ell\cdot\nabla\zeta_\varepsilon\,dy
+\int_{B_{r_0}}|\cdot|^{-4}m_K\zeta_\varepsilon^2\phi^2A_\ell^2\,dy\notag\\
&\quad=2\int_{B_{r_0}}|\cdot|^{-2}\rho_K\zeta_\varepsilon^2\phi^2A_\ell P\cdot\nabla B_\ell\,dy.
\end{align}
From H\"older's inequality and  the assumption $|\nabla\zeta_\varepsilon|\le C/\varepsilon$, we infer
that
\begin{align*}
\left|2\int_{B_{r_0}}\zeta_\varepsilon\phi^2A_\ell\nabla A_\ell\cdot\nabla\zeta_\varepsilon\,dy\right|
\le2\left\|\phi\nabla A_\ell\right\|_{L^2(B_{2\varepsilon}\setminus B_\varepsilon;\mathbb{R}^2)}
\left\|\phi A_\ell\nabla\zeta_\varepsilon\right\|_{L^2(B_{2\varepsilon}\setminus B_\varepsilon;\mathbb{R}^2)}
\end{align*}
and
\begin{align*}
\|\phi A_\ell\nabla\zeta_\varepsilon\|_{L^2(B_{2\varepsilon}\setminus B_\varepsilon;\mathbb{R}^2)}
\le C\varepsilon^{-1}\left(\int_{B_{2\varepsilon}\setminus B_\varepsilon}A_\ell^2\,dy\right)^{\frac12}
\le C\varepsilon\left(\int_{B_{2\varepsilon}\setminus B_\varepsilon}|y|^{-4}A_\ell^2\,dy\right)^{\frac12},
\end{align*}
which, together with the proved fact that $|\cdot|^{-4}A_\ell^2\in L^1(B_{r_0};\mathbb{R})$
[see \eqref{eq:Kelvin-energy}], further implies that
\begin{align}\label{e5.6}
\lim_{\varepsilon\to0}2\int_{B_{r_0}}\zeta_\varepsilon\phi^2A_\ell\nabla A_\ell\cdot\nabla\zeta_\varepsilon\,dy=0.
\end{align}
Applying \eqref{e5.6} and the dominated convergence theorem, and letting $\varepsilon\to0$ in \eqref{e5.5},
we then conclude that \eqref{eq:mass-test} holds. This finishes the proof of Lemma \ref{lem:point-removal}.
\end{proof}

Let $R_2\in(0,\infty)$ be as in Proposition \ref{prop:tail-energy}, $r_0:=R_1^{-1}$,
and $A_\ell$ and $B_\ell$ with $\ell\in\{1,2\}$ be as in \eqref{eq:Kelvin-unknowns}. For any $r\in(0,r_0)$, let
\begin{align}\label{e5.x1}
E_A(r):=\sum_{\ell=1}^2\int_{B_r}\left(|\nabla A_\ell|^2+|y|^{-4}A_\ell^2\right)\,dy,\quad
E_B(r):=\sum_{\ell=1}^2\int_{B_r}|\nabla B_\ell|^2\,dy,
\end{align}
\begin{align}\label{e5.x2}
E(r):=E_A(r)+E_B(r),
\end{align}
\begin{align}\label{e5.x3}
\delta(r):=\sup_{y\in B_r}|P(y)|,\quad\text{and}\quad
\omega(r):=\sup_{y\in B_r}\left|1-\rho_K(y)^2\right|.
\end{align}
It is easy to find that both $\delta(r)$ and $\omega(r)$ are nondecreasing in $r$.
Meanwhile, by \eqref{eq:Kelvin-values-zero}, we conclude that
\begin{equation}\label{eq:delta-omega-zero}
\delta(r)\rightarrow0\quad\text{and}\quad \omega(r)\rightarrow0\quad\text{as}\quad r\to0.
\end{equation}

\begin{lemma}\label{lem:A-improvement}
Let $E_A$ and $E_B$ and $\delta$ be respectively as in
\eqref{e5.x1} and \eqref{e5.x3}. Fix $\vartheta\in(0,1/4)$. Then there exist positive constants $r_A$ and
$C_A:=C_A(\vartheta)$ such that, for any $r\in(0,r_A)$,
\begin{equation}\label{eq:A-improvement}
E_A(\vartheta r)\le C_A[\delta(r)]^2E_B(r)+C_Ar^2E_A(r).
\end{equation}
\end{lemma}

\begin{proof}
Since $m_K(r)\to2$ and $\rho_K(r)\to1$ as $r\to0$, it follows that there exists $r_A\in(0,r_0)$
such that $m_K\ge1$ and $1/2\le\rho_K\le1$ on $B_{r_A}$.
For any given $r\in(0,r_A)$, take $\phi\in C_{\rm c}^\infty(B_r;\mathbb{R})$ such that $0\le\phi\le1$, $\phi\equiv1$ on
$B_{\vartheta r}$, and $|\nabla\phi|\le C_\vartheta/r$.

Then, by Young's inequality and the proved fact that $1/2\le\rho_K\le1$ on $B_{r_A}$, we conclude that,
on $B_{r_A}$,
$$
2|\cdot|^{-2}\rho_K|A_\ell||P||\nabla B_\ell|\le\frac14|\cdot|^{-4}A_\ell^2+C[\delta(r)]^2|\nabla B_\ell|^2
$$
and
$$
2\left|\phi A_\ell\nabla A_\ell\cdot\nabla\phi\right|\le\frac14\phi^2|\nabla A_\ell|^2+CA_\ell^2|\nabla\phi|^2.
$$
From this, the proved fact that $m_K\ge1$ on $B_{r_A}$, and \eqref{eq:mass-test}, we infer that, any $r\in(0,r_A)$,
\begin{equation}\label{e6.1}
\int_{B_{\vartheta r}}\left(|\nabla A_\ell|^2+|\cdot|^{-4}A_\ell^2\right)\,dy
\le C[\delta(r)]^2\int_{B_r}|\nabla B_\ell|^2\,dy+C_\vartheta r^{-2}\int_{B_r}A_\ell^2\,dy.
\end{equation}
Since $|y|^4\le r^4$ for  $y\in B_r$, it follows that $r^{-2}\int_{B_r}A_\ell^2\,dy\le r^2\int_{B_r}|\cdot|^{-4}A_\ell^2\,dy$.
This, combined with \eqref{e6.1}, implies that \eqref{eq:A-improvement} holds. This finishes the proof of
Lemma \ref{lem:A-improvement}.
\end{proof}

\begin{lemma}\label{lem:harmonic-replacement}
Let $r\in(0,\infty)$ and $H\in H^1(B_r;\mathbb{R})$ be weakly harmonic, i.e.,
$\int_{B_r}\nabla H\cdot\nabla\varphi\,dx=0$ for any
$\varphi\in H_0^1(B_r;\mathbb{R})$.  Then, for any $\vartheta\in(0,1)$,
\begin{equation}\label{eq:harmonic-scaling}
\int_{B_{\vartheta r}}|\nabla H|^2\,dy\le\vartheta^2\int_{B_r}|\nabla H|^2\,dy.
\end{equation}
Moreover, for any $g\in H^1(B_r;\mathbb{R})$, there exists a unique weakly harmonic function
$H_g\in g+H_0^1(B_r;\mathbb{R})$ such that
\begin{equation}\label{eq:harmonic-minimization}
\int_{B_r}|\nabla H_g|^2\,dy\le\int_{B_r}|\nabla g|^2\,dy
\quad\text{and}\quad\int_{B_r}\nabla H_g\cdot\nabla(g-H_g)\,dy=0.
\end{equation}
\end{lemma}

The estimate \eqref{eq:harmonic-scaling} is well known (see, for instance, \cite[Lemma 3.10]{hl11}).
Moreover, the existence and uniqueness of the weakly harmonic function $H_g$ and the estimate
\eqref{eq:harmonic-minimization} can be found in \cite[pp.\,31-32]{ks00}.

\begin{lemma}\label{lem:B-improvement}
Let $E_A$, $E_B$, $\omega$, and $\delta$ be as in \eqref{e5.x1} and \eqref{e5.x3}, respectively. Fix $\vartheta\in(0,1/4)$.
Then there exist positive constants $r_B$ and
$C_B:=C_B(\vartheta)$ such that, for any $r\in(0,r_B)$,
\begin{equation}\label{eq:B-improvement}
E_B(\vartheta r)\le2\vartheta^2E_B(r)+C_B[\omega(r)]^2E_B(r)+C_B[\delta(r)]^2E_A(r).
\end{equation}
\end{lemma}

\begin{proof}
Since $\rho_K(r)\to1$ as $r\to0$, it follows that there exists $r_B\in(0,r_0)$
such that $1/2\le\rho_K\le1$ on $B_{r_B}$.

Fix $r\in(0,r_B)$ and $\ell\in\{1,2\}$. Let $H_\ell$ be the harmonic replacement of $B_\ell$ in
$B_r$, i.e., $H_\ell\in B_\ell+H_0^1(B_r;\mathbb{R})$,
$$\Xi_\ell:=B_\ell-H_\ell\in H_0^1(B_r;\mathbb{R}),\quad\text{and}\quad
G_\ell:=2|\cdot|^{-2}\rho_KA_\ell P.
$$
By Lemma~\ref{lem:point-removal}, we find that $\Xi_\ell$ is an admissible test function to
$\mathrm{div}(\rho_K^2\nabla B_\ell-G_\ell)=0$.  Since $\int_{B_r}\nabla H_\ell\cdot\nabla\Xi_\ell\,dx=0$,
it follows that
$$\int_{B_r}\rho_K^2|\nabla\Xi_\ell|^2\,dx=-\int_{B_r}\left(\rho_K^2-1\right)\nabla H_\ell\cdot\nabla\Xi_\ell\,dx
+\int_{B_r}G_\ell\cdot\nabla\Xi_\ell\,dx.
$$
By this, the proved fact that $1/2\le\rho_K\le1$ on $B_{r_B}$, and Young's inequality, we conclude that
\begin{equation}\label{eq:Xi-estimate}
\int_{B_r}|\nabla\Xi_\ell|^2\,dx\le C[\omega(r)]^2\int_{B_r}|\nabla H_\ell|^2\,dx
+C\int_{B_r}|G_\ell|^2\,dx.
\end{equation}
Moreover, from \eqref{eq:harmonic-minimization} and the definition of $G_\ell$, we deduce that
\begin{equation}\label{e6.2}
\int_{B_r}|\nabla H_\ell|^2\,dx\le\int_{B_r}|\nabla B_\ell|^2\,dx\quad\text{and}\quad
\int_{B_r}|G_\ell|^2\,dx\le C[\delta(r)]^2\int_{B_r}|\cdot|^{-4}A_\ell^2\,dx.
\end{equation}
Therefore, applying \eqref{eq:harmonic-scaling}, \eqref{eq:Xi-estimate}, and \eqref{e6.2}, we find that
\begin{align}\label{e6.3}
\int_{B_{\vartheta r}}|\nabla B_\ell|^2\,dx&\le2\int_{B_{\vartheta r}}|\nabla H_\ell|^2\,dx
+2\int_{B_r}|\nabla\Xi_\ell|^2\,dx\notag\\
&\le2\vartheta^2\int_{B_r}|\nabla B_\ell|^2\,dx+C[\omega(r)]^2\int_{B_r}|\nabla B_\ell|^2\,dx
+C[\delta(r)]^2\int_{B_r}|\cdot|^{-4}A_\ell^2\,dx.
\end{align}
Then, summing over $\ell$ in \eqref{e6.3}, we obtain \eqref{eq:B-improvement}.
This finishes the proof of Lemma \ref{lem:B-improvement}.
\end{proof}

\begin{proposition}\label{prop:Morrey}
Let $R_2\in(0,\infty)$ be as in Proposition \ref{prop:tail-energy}, $r_0:=R_2^{-1}$, and $E$ be as in \eqref{e5.x2}.
Then, for any given $\beta\in(0,1)$, there exist $r_\beta\in(0,r_0)$ and $C_\beta\in(0,\infty)$ such that,
for any $r\in(0,r_\beta]$,
\begin{equation}\label{eq:Morrey}
E(r)\le C_\beta r^{2\beta}.
\end{equation}
\end{proposition}

\begin{proof}
Fix $\beta\in(0,1)$. Choose $\vartheta\in(0,1/4)$ such that
\begin{equation}\label{eq:theta-choice}
2\vartheta^2\le\frac14\vartheta^{2\beta}.
\end{equation}
By \eqref{eq:delta-omega-zero}, we find that there exists $r_\beta\in(0,\min\{r_A,r_B\})$ such that, for any $r\in(0,r_\beta]$,
\begin{align}
C_Ar^2+C_B[\delta(r)]^2&\le\frac12\vartheta^{2\beta}\label{eq:small-A}
\end{align}
and
\begin{align}
C_A[\delta(r)]^2+C_B[\omega(r)]^2&\le\frac14\vartheta^{2\beta},\label{eq:small-B}
\end{align}
where the constants $r_A$ and $r_B$ are respectively as in Lemmas \ref{lem:A-improvement} and \ref{lem:B-improvement}.
Then, applying \eqref{eq:A-improvement}, \eqref{eq:B-improvement},
\eqref{eq:theta-choice}, \eqref{eq:small-A}, and \eqref{eq:small-B}, we find that, for any $r\in(0, r_\beta]$,
\begin{equation}\label{eq:one-step}
E(\vartheta r)\le\vartheta^{2\beta}E(r).
\end{equation}
For any $k\in\mathbb{N}$, let $r_k:=\vartheta^kr_\beta$. From
\eqref{eq:one-step} and an iteration argument, it follows that,
for any $k\in\mathbb{N}$, $E(r_k)\le\vartheta^{2\beta k}E(r_\beta)$.
This, combined with the monotonicity of $E$, implies that, for any $r\in(0,r_\beta]$
satisfying $r\in(r_{k+1}, r_k]$ for some $k\in\mathbb{N}$,
$$
E(r)\le E(r_k)\le\vartheta^{-2\beta}r_\beta^{-2\beta}E(r_\beta)r^{2\beta}.
$$
Thus, \eqref{eq:Morrey} holds with $C_\beta:=\vartheta^{-2\beta}r_\beta^{-2\beta}E(r_\beta)$. This finishes
the proof of Proposition \ref{prop:Morrey}.
\end{proof}

\section{\texorpdfstring{$L^4$}--integrability of the phase form}\label{s7}
In this section, we prove the  $L^4$-integrability of the phase form $p$ on an exterior domain
by applying the Morrey-type decay estimate established in Proposition \ref{prop:Morrey}.
Precisely, let $R_2\in(0,\infty)$ be as in Proposition \ref{prop:tail-energy}, $r_0:=R_2^{-1}$,
$p$ and $I$ be respectively as in \eqref{eq:rho-p-def} and \eqref{eq:Kelvin-def}.
Define the \emph{Kelvin phase vector} $B$ by setting, for any $x\in B_{r_0}$,
\begin{equation}\label{eq:B-vector}
B(x):=(B_1(x),B_2(x)):=p(I(x))
\end{equation}
and its \emph{disk average} $M(r)$ with $r\in(0,r_0)$ by setting
\begin{equation*}
M(r):=\frac1{\pi r^2}\int_{B_r}B(x)\,dx.
\end{equation*}
By Lemma~\ref{lem:point-removal}, we find that $B\in H^1(B_{r_0};\mathbb{R}^2)$ and its continuous extension
satisfies $B({\bf0})={\bf0}$.

\begin{lemma}\label{lem:B-L2}
Let $R_2\in(0,\infty)$ be as in Proposition \ref{prop:tail-energy}, $r_0:=R_2^{-1}$, and $B$ be as in \eqref{eq:B-vector}.
Then, for any given $\beta\in(0,1)$, there exist $r_\beta\in(0,r_0)$ and $C_\beta\in(0,\infty)$ such that,
for any $r\in(0,r_\beta]$,
\begin{equation}\label{eq:B-L2}
\int_{B_r}|B|^2\,dy\le C_\beta r^{2+2\beta}.
\end{equation}
\end{lemma}

To show Lemma \ref{lem:B-L2}, we need the following Poincar\'e inequality
(see, for instance, \cite[Theorem 13.36]{l17}).

\begin{lemma}\label{lem:Poincare-disk}
Let $r\in(0,\infty)$, $f\in H^1(B_r;\mathbb{R}^2)$, and
$f_{B_r}:=|B_r|^{-1}\int_{B_r}f\,dy$.  Then
\begin{equation*}
\int_{B_r}|f-f_{B_r}|^2\,dx\le Cr^2\int_{B_r}|\nabla f|^2\,dx,
\end{equation*}
where $C$ is a positive constant independent of both $f$ and $r$.
\end{lemma}

\begin{proof}[Proof of Lemma \ref{lem:B-L2}]
By Lemma~\ref{lem:Poincare-disk} and Proposition~\ref{prop:Morrey},  we find that, for any
given $\beta\in(0,1)$, there exist $r_\beta\in(0,r_0)$ and $C_\beta\in(0,\infty)$ such that,
for any $r\in(0,r_\beta]$,
\begin{equation}\label{eq:B-oscillation}
\int_{B_r}|B-M(r)|^2\,dy\le Cr^2\int_{B_r}|\nabla B|^2\,dy\le C_\beta r^{2+2\beta},
\end{equation}
which further implies that, for any $r\in(0,r_\beta]$,
\begin{equation}\label{e7.1}
|M(r)-M(r/2)|\le |B_{r/2}|^{-\frac12}\left[\int_{B_r}|B-M(r)|^2\,dy\right]^{\frac12}
\le C_\beta r^\beta.
\end{equation}
Since $B(y)\to{\bf0}$ uniformly as $y\to{\bf0}$, it follows that, for any given $r\in(0,r_\beta]$,
$M(2^{-k}r)\to0$ as $k\to\infty$.  By this and \eqref{e7.1}, we conclude that, for any $r\in(0,r_\beta]$,
$$|M(r)|\le\sum_{k=0}^\infty\left|M\left(2^{-(k-1)}r\right)-M\left(2^{-k}r\right)\right|
\le\sum_{k=0}^\infty C_\beta (2^{-k+1}r)^\beta=C_\beta r^\beta,$$
which, combined with \eqref{eq:B-oscillation}, implies that \eqref{eq:B-L2} holds.
This finishes the proof of Lemma \ref{lem:B-L2}.
\end{proof}

For any given $R\in(0,\infty)$, define the annuli $\mathcal A_R$ and $\mathcal A_R^*$, respectively, by setting
\begin{equation*}
\mathcal A_R:=\{x\in\mathbb{R}^2:R<|x|<2R\}\quad\text{and}
\quad\mathcal A_R^*:=\{x\in\mathbb{R}^2:R/2<|x|<4R\}.
\end{equation*}

\begin{lemma}\label{lem:annular-estimates}
Let $p$ be as in \eqref{eq:rho-p-def}. Fix $\beta\in(0,1)$. Then there exists a positive constant $C_\beta$
such that, for sufficiently large $R\in(0,\infty)$,
\begin{align}
\int_{\mathcal A_R^*}|\nabla p|^2\,dx\le C_\beta R^{-2\beta}\label{eq:Dp-annulus}
\end{align}
and
\begin{align}
\int_{\mathcal A_R^*}|p|^2\,dx&\le C_\beta R^{2-2\beta}.\label{eq:p2-annulus}
\end{align}
\end{lemma}

\begin{proof}
Let $r_\beta$ be as in Proposition \ref{prop:Morrey}. Take $R$ sufficiently large such that $2/R< r_\beta$.
Note that $I$ maps $\mathcal A_R^*$ onto $\{1/(4R)<|y|<2/R\}$.
From this and Proposition \ref{prop:Morrey}, we infer that
$$\int_{\mathcal A_R^*}|\nabla p|^2\,dx\le E_B(2/R)\le C_\beta R^{-2\beta}.
$$
Similarly, applying Lemma~\ref{lem:B-L2}, we find that
$$\int_{\mathcal A_R^*}|p|^2\,dx\le(4R)^4\int_{B_{2/R}}|B(y)|^2\,dy\le C_\beta R^{2-2\beta}.
$$
This finishes the proof of Lemma \ref{lem:annular-estimates}.
\end{proof}

\begin{lemma}\label{lem:annular-GN}
Let $R\in(0,\infty)$. Then there exists a positive constant $C$ independent of $R$ such that,
for any $f\in H^1(\mathcal A_R^*;\mathbb{R}^2)$,
\begin{equation}\label{eq:annular-GN}
\int_{\mathcal A_R}|f|^4\,dx\le C\left(\int_{\mathcal A_R^*}|f|^2\,dx\right)\left(\int_{\mathcal A_R^*}|\nabla f|^2\,dx
+R^{-2}\int_{\mathcal A_R^*}|f|^2\,dx\right).
\end{equation}
\end{lemma}

\begin{proof}
Let $R\in(0,\infty)$ and $f\in H^1(\mathcal A_R^*;\mathbb{R}^2)$. For any $x\in\mathbb{R}^2$, let $F(x):=f(Rx)$ and
$$U:=\left\{z\in\mathbb{R}^2:1/2<|z|<4\right\}.$$
 Let $\varphi\in C_{\rm c}^\infty(U;\mathbb{R})$
satisfy $\varphi=1$ on $\{z\in\mathbb{R}^2:1<|z|<2\}$. Then $\varphi F\in H_0^1(U;\mathbb{R}^2)$
and hence the zero extension of $\varphi F$ on $\mathbb{R}^2$ belongs to $H^1(\mathbb{R}^2;\mathbb{R}^2)$.
From the Gagliardo--Nirenberg interpolation inequality (see, for instance, \cite[Theorem 12.83]{l17}),
we deduce that there exists a positive constant $C$, independent of both $F$ and $\varphi$, such that
\begin{equation}\label{e7.1x}
\|\varphi F\|_{L^4(\mathbb{R}^2;\mathbb{R}^2)}^4 \le C\|\varphi F\|_{L^2(\mathbb{R}^2;\mathbb{R}^2)}^2
\|\nabla(\varphi F)\|_{L^2(\mathbb{R}^2;\mathbb{R}^2)}^2.
\end{equation}
Moreover, by the choice of $\varphi$, we conclude that
\begin{equation}\label{e7.1x1}
\|\nabla(\varphi F)\|_{L^2(\mathbb{R}^2;\mathbb{R}^2)}^2
\le C\left[\|\nabla F\|_{L^2(U;\mathbb{R}^2)}^2+\|F\|_{L^2(U;\mathbb{R}^2)}^2\right].
\end{equation}
Then, applying \eqref{e7.1x}, \eqref{e7.1x1}, and a scaling argument, we find that
\eqref{eq:annular-GN} holds. This finishes the proof of Lemma \ref{lem:annular-GN}.
\end{proof}

\begin{proposition}\label{prop:p-L4}
Let $R_2\in(0,\infty)$ be as in Proposition \ref{prop:tail-energy} and $p$ be as in \eqref{eq:rho-p-def}.
Then there exists $R_3\in[R_2,\infty)$ such that $p\in L^4(\Omega_{R_3};\mathbb{R}^2)$.
\end{proposition}

\begin{proof}
Fix $\beta\in(1/2,1)$ and let $r_\beta$ be as in Proposition \ref{prop:Morrey}.
Take $R$ sufficiently large such that $2/R< r_\beta$.
Applying \eqref{eq:annular-GN} to $f:=p$ and using \eqref{eq:Dp-annulus} and \eqref{eq:p2-annulus},
we conclude that
\begin{equation}\label{e7.2}
\int_{\mathcal A_R}|p|^4\,dx\le C_\beta R^{2-4\beta}.
\end{equation}
Choose $R_3\in[R_2,\infty)$ so large that the estimate \eqref{e7.2} holds for any $R:=2^kR_3$
with $k\in\mathbb{Z}_+$. Note that the dyadic annuli $\{\mathcal A_{2^kR_3}\}_{k=0}^\infty$
cover $\Omega_{R_3}$ up to boundaries and $2-4\beta<0$. From this and \eqref{e7.2}, it follows that
$$
\int_{\Omega_{R_3}}|p|^4\,dx\le C_\beta\sum_{k=0}^\infty(2^kR_3)^{2-4\beta}<\infty.
$$
This shows $p\in L^4(\Omega_{R_3};\mathbb{R}^2)$ and hence finishes the proof of Proposition
\ref{prop:p-L4}.
\end{proof}

\section{Proof of Proposition \ref{prop:h-L2}}\label{s8}
In this section, we give the proof of Proposition \ref{prop:h-L2}. We begin by establishing the following
lemma.

\begin{lemma}\label{lem:defect-equation}
Let $R_1$ be as in \eqref{eq:mass-Schur}, $u\in C^\infty(\mathbb{R}^2;\mathbb{C})$ be a solution
to \eqref{eq:GL} having the asymptotic property \eqref{eq:modulus-limit},
and $\rho:=|u|$. Then $h:=1-\rho$ satisfies
\begin{equation}\label{eq:defect-equation}
\left(-\Delta+\rho(1+\rho)\right)h=\rho|p|^2\quad\hbox{on }\Omega_{R_1}.
\end{equation}
Moreover,
\begin{equation}\label{eq:positive-mass}
\rho(1+\rho)\ge\frac34\quad\hbox{on }\Omega_{R_1}.
\end{equation}
\end{lemma}

\begin{proof}
Since $\rho=1-h$, it follows that the equation \eqref{eq:amplitude} becomes
$\Delta h+\rho|p|^2=\rho(1+\rho)h$, which is precisely \eqref{eq:defect-equation}.
Moreover, \eqref{eq:positive-mass} directly follows from \eqref{eq:tail-smallness}.
This finishes the proof of Lemma \ref{lem:defect-equation}.
\end{proof}

We now prove Proposition \ref{prop:h-L2} by using Lemma \ref{lem:defect-equation} and
Propositions~\ref{prop:tail-energy}, \ref{lem:maxmod}, and \ref{prop:p-L4}.

\begin{proof}[Proof of Proposition \ref{prop:h-L2}]
By Proposition~\ref{prop:tail-energy} and $a_\ell=\partial_\ell h$ with $\ell\in\{1,2\}$,
we find that $\nabla h\in L^2(\Omega_{R_2};\mathbb{R}^2)$, where $R_2\in(0,\infty)$ is as in Proposition \ref{prop:tail-energy}.
Choose a radial function $\phi_1\in C^\infty(\mathbb{R}^2;\mathbb{R})$ such that $0\le\phi_1\le1$,
$\phi\equiv0$ on $B_{R_2+1/2},$ and $\phi\equiv1$ on $\Omega_{R_2+1}$.  For any given $L\in(2(R_2+1),\infty)$,
choose another radial function $\phi_2\in C_{\rm c}^\infty(\mathbb{R}^2;\mathbb{R})$ such that
$0\le\phi_2\le1$, $\phi_2\equiv1$ on $B_L$, $\phi_2\equiv0$ on $\Omega_{2L}$, and $|\nabla\phi_2|\le C/L$.
For any given $L\in(2(R_2+1),\infty)$, let
$\phi_L:=\phi_1\phi_2$. Then $\phi_L\in C_{\rm c}^\infty(\Omega_{R_2};\mathbb{R})$.

For any given $L\in(2(R_2+1),\infty)$, using $\phi_L^2h$ as a test function in \eqref{eq:defect-equation},
we find that
\begin{align}\label{e8.1}
&\int_{\Omega_{R_2}}\phi_L^2|\nabla h|^2\,dx+2\int_{\Omega_{R_2}}\phi_Lh\nabla h\cdot\nabla\phi_L\,dx
+\int_{\Omega_{R_2}}\rho(1+\rho)\phi_L^2h^2\,dx\notag\\
&\quad=\int_{\Omega_{R_2}}\rho\phi_L^2h|p|^2\,dx.
\end{align}
From \eqref{eq:positive-mass}, the proved fact $0<\rho\le1$ in Proposition \ref{lem:maxmod}, and Young's inequality,
we deduce that
\begin{align*}
2\left|\phi_Lh\nabla h\cdot\nabla\phi_L\right|&\le
\frac14\phi_L^2|\nabla h|^2+Ch^2|\nabla\phi_L|^2\quad\text{and}\quad
\rho\phi_L^2h|p|^2\le\frac38\phi_L^2h^2+C\phi_L^2|p|^4
\end{align*}
on $\Omega_{R_2}$.
This, combined with \eqref{e8.1}, further implies that
\begin{equation}\label{eq:defect-Caccioppoli}
\int_{\Omega_{R_2}}\phi_L^2\left(|\nabla h|^2+h^2\right)\,dx\le C\int_{\Omega_{R_2}}\phi_L^2|p|^4\,dx
 +C\int_{\Omega_{R_2}}h^2|\nabla\phi_L|^2\,dx.
\end{equation}
From the fact that $h(x)\to0$ uniformly as $|x|\to\infty$
and the assumptions for both $\phi_1$ and $\phi_2$ as above,
it follows that
\begin{equation}\label{e8.3}
\int_{\Omega_{R_2}}h^2|\nabla\phi_L|^2\,dx=\int_{L<|x|<2L}h^2|\nabla\phi_L|^2\,dx
\le C\sup_{B_{2L}\backslash B_L}h^2\rightarrow0
\end{equation}
as $L\to\infty.$
Moreover, by Proposition~\ref{prop:p-L4} and the proved fact that
$|p|$ is bounded on $\Omega_{R_1}$, we conclude that
\begin{align}\label{e8.4}
\int_{\Omega_{R_2}}|p|^4\,dx<\infty.
\end{align}
Meanwhile, it is easy to find that, for any $x\in\mathbb{R}^2$,
$$[\varphi_L(x)]^2\left(|\nabla h(x)|^2+[h(x)]^2\right)
\to\left(|\nabla h(x)|^2+[h(x)]^2\right)\quad{as}\quad L\to\infty,
$$
which, together with Fatou's lemma, \eqref{eq:defect-Caccioppoli},
\eqref{e8.3}, and \eqref{e8.4}, further implies that
\begin{align*}
\int_{\Omega_{R_2}}\left(|\nabla h|^2+h^2\right)\,dx
&\le\liminf_{L\to\infty}\int_{\Omega_{R_2}}\phi_L^2\left(|\nabla h|^2+h^2\right)\,dx\\
&\le\liminf_{L\to\infty}\left(C\int_{\Omega_{R_2}}\phi_L^2|p|^4\,dx
 +C\int_{\Omega_{R_2}}h^2|\nabla\phi_L|^2\,dx\right)\\
&\le C\int_{\Omega_{R_2}}|p|^4\,dx<\infty.
\end{align*}
Thus, $h\in L^2(\Omega_{R_2};\mathbb{R})$ and hence \eqref{eq:h-L2} holds. This finishes the
proof of Proposition \ref{prop:h-L2}.
\end{proof}

\smallskip

\noindent\textbf{Acknowledgements}\quad
The authors acknowledge the use of AI tools during the exploratory stage of this project.
All mathematical arguments and proofs in the final manuscript were checked
and written by the authors. The authors would also like to thank 
Professor Juncheng Wei for bringing to their attention
his preprint with Hongge Chen, Haicheng Yan, Wen Yang, arXiv:2607.17490 [math.AP],
in which Professor Wei et al., independently and simultaneously,
obtained a closely related result. Professor Wei also
informed us that he had already announced their result in a talk
at the Summer School of the Harbin Engineering 
University on July 16, 2026. 
We became aware of their work
only after the present manuscript had already been submitted to the arXiv.
We also mention that the first version of our article was
submitted to arXiv on July 18, 2026, and then this article
was on hold by arXiv till September 17, 2026. During the period
when it was on hold, on August 13 we submitted a revised version
to arXiv by adding the above acknowledgement to AI.

%
%
%

\bigskip

\noindent
Xiaosheng Lin

\smallskip

\noindent
School of Mathematical Sciences, Jimei University,
Xiamen 361005, The People's Republic of China

\smallskip

\noindent {\it E-mail}: \texttt{xslin@jmu.edu.cn}

\bigskip

\noindent Dachun Yang (Corresponding author), Wen Yuan and Yangyang Zhang

\smallskip

\noindent Laboratory of Mathematics and Complex Systems
(Ministry of Education of China),
School of Mathematical Sciences, Institute for Advanced Study,
Beijing Normal University,
Beijing 100875, The People's Republic of China

\smallskip

\noindent{\it E-mails:} \texttt{dcyang@bnu.edu.cn} (D. Yang)

\noindent\phantom{{\it E-mails:}} \texttt{wenyuan@bnu.edu.cn} (W. Yuan)

\noindent\phantom{{\it E-mails:}} \texttt{yangyzhang@bnu.edu.cn} (Y. Zhang)

\bigskip

\noindent Sibei Yang

\medskip

\noindent School of Mathematics and Statistics, Lanzhou University,
Lanzhou 730000, The People's Republic of China

\smallskip

\noindent{\it E-mail:} \texttt{yangsb@lzu.edu.cn}

\end{document}